\documentclass[12pt, reqno]{amsart}
\usepackage{amsmath, amsthm, amscd, amsfonts, amssymb, graphicx, color}
\usepackage[bookmarksnumbered, colorlinks, plainpages]{hyperref}
\hypersetup{colorlinks=true,linkcolor=red, anchorcolor=green, citecolor=cyan, urlcolor=red, filecolor=magenta, pdftoolbar=true}

\newtheorem{theorem}{Theorem}[section]
\newtheorem{lemma}[theorem]{Lemma}

\theoremstyle{definition}
\newtheorem{definition}[theorem]{Definition}

\theoremstyle{remark}

\numberwithin{equation}{section}

\begin{document}
\setcounter{page}{1}

\title[Supporting functionals]{Characterization of Supporting Functionals at Points of the Unit Sphere of Orlicz--Lorentz Spaces}

\author[D. Wang, Y. Li]{Di Wang$^1$ and Yongjin Li$^2$$^{*}$}

\address{$^{1}$ Department of Mathematics, Sun Yat-sen University, Guangzhou 510275, China.}
\email{\textcolor[rgb]{0.00,0.00,0.84}{wangd267@mail2.sysu.edu.cn}}

\address{$^{2}$ Department of Mathematics, Sun Yat-sen University, Guangzhou 510275, China.}
\email{\textcolor[rgb]{0.00,0.00,0.84}{stslyj@mail.sysu.edu.cn}}




\date{Received: xxxxx; Revised: yyyyyy; Accepted: zzzzzz.
\newline \indent $^{*}$ Corresponding author}

\begin{abstract}
In this paper we give a complete characterization of the supporting functionals at any point on the unit sphere of Orlicz--Lorentz spaces $\Lambda_{\varphi, \omega}$. Departing from traditional approaches, we establish our results without assuming that the Orlicz function $\varphi$ is an N--function. These results provide a basis for studying the extremal structures of Orlicz--Lorentz spaces.
\end{abstract} \maketitle

\section{Introduction and preliminaries}
The main purpose of this article is to provide a complete characterization of the supporting functionals at any point on the unit sphere of Orlicz--Lorentz spaces $\Lambda_{\varphi, \omega}^{o}$ generated by an arbitrary Orlicz function $\varphi$ and a decreasing weight $\omega$. The study of supporting functionals is not only a central topic in the geometry of Banach spaces but also serves as an essential tool for investigating various extremal structures. Recently, in \cite{Wang2024}, the descriptions of supporting functionals in these spaces were established; however, the results rely on the assumption that the Orlicz function $\varphi$ is an N-function.
In this article, we remove the N-function restriction, provide a more general framework for the study of supporting functionals in Orlicz--Lorentz spaces.

Let $X$ be a Banach space and let $x\in S(X)$ where $S(X)$ is the unit sphere of $X$. A functional $f\in X^{*}$ is called a supporting functional at $x$ if 
$$f(x)=1, \;\|f\|_{X^{*}}=1.$$
Equivalently, $f$ attains its norm at $x$ and $\|f\|_{X^{*}}=1$.

For $f\in L_{0}$, its non-increasing rearrangement $f^*$ is defined by
\begin{equation*}
	\begin{aligned}
f^{*}(t):&=\inf\{\lambda >0 : \mu_{f}(\lambda)\leq t \}, \;  t>0,\\
f^{*}(0):&=\lim_{t\rightarrow 0+}f^{*}(t)=ess\: sup|f|.
\end{aligned}
\end{equation*}
where $\mu_{f}$ is the distribution function of $f$, and it is defined by
$$
\mu_{f}(t):=\mu(\{s \in R^{+} : | f(s)|  > t\}), \;  t>0.
$$
$\sigma:(0,r)\rightarrow (0,r)$ $(r\leq \infty)$ is called a measure-preserving transformation (\cite{Bennett1953,Brudnyi19882009}) if for every measurable set $E\subset (0,r)$, $\sigma^{-1}(E) $ is measurable and $\mu(\sigma^{-1}(E))=\mu(E)$.
It was given in \cite[Corollary 7.6]{Bennett1953} that for a resonant measure space $(R, \mu)$ and a non--negative $\mu$-measurable function $f$ on $R$ satisfying $\lim_{t\rightarrow\infty}f^{*}(t)=0$, there exists a measure-preserving transformation $\sigma$ from the support of $f$ onto the support of $f^{*}$ such that $f=f^{*}\circ\sigma\;\mu-a.e.$ on the support of $f$. For more results on measure-preserving transformations, we refer the reader to \cite{Bennett1953,Ryff1970449}.

For a function $\omega: (0,\infty) \rightarrow (0, \infty)$, If $\omega$ is non-increasing, locally integrable with respect to the Lebesgue measure $\mu$ and
$\int_{0}^{\infty}\omega(t)dt=\infty$, then $\omega$ is called a weight function.
The Lorentz space $\Lambda_{\omega}$ is defined as
\begin{equation*}
	\Lambda_{\omega}=\left\{ f\in L_{0}: \|f\|_{\Lambda_{\omega}}=\int_{R^{+}}f^{*}(t)\omega(t)dt=\int_{R^{+}}f^{*}(t)dW<\infty \right\},
\end{equation*}
and the Marcinkiewicz space $M_{W}$ is defined as
\begin{equation*}
	M_{W}=\left\{ f\in L^{0}: \|f\|_{M_{W}}=\sup_{\alpha>0}\frac{\int_{0}^{\alpha}f^{*}(t)dt}{W(\alpha)}<\infty \right\}.
\end{equation*}
where 
$
W(t):=\int_{0}^{t}\omega(s)ds, \; t\geq0.
$
The two spaces are K\"{o}the dual to each other.

For $f\in L_{0}$ and $g\in L_{0}$, if $\int_{0}^{t}f^{*}(s)ds\leq\int_{0}^{t}g^{*}(s)ds$ for $t>0$, then we say $f$ is submajorized by $g$, and we denote this by $f\prec g$.
Obviously $\|f\|_{M_{W}}\leq 1$ whenever $f\prec \omega$.

A function $[0, \infty]\rightarrow [0, \infty]$ is an Orlicz function(\cite{Chen1996}) if $\varphi$ is convex, $\varphi(0)=0$ and $\varphi(t)>0$ for all $t>0$. For any Orlicz function $\varphi$, its complementary function $\psi$ in the sense of Young is defined by the formula
$$
\psi(t)=\sup_{s>0}\{| st | -\varphi(s)\}
$$
where $t>0$. Obviously, if $\varphi$ is an Orlicz function, then $\psi$ is an Orlicz function as well. 
We denote the right--derivative of $\varphi$ and $\psi$ by $p$, $q$, respectively. Both $p(s)$ and $q(t)$ are non--decreasing and satisfy:
$$
\varphi(u)=\int_{0}^{u}p(t)dt, \quad \psi(v)=\int_{0}^{v}q(s)ds \quad u\geq0, v\geq0.
$$
$\varphi$ and $\psi$ satisfy Young's Inequality:
$$
uv \leq \varphi(u)+ \psi(v), \quad u\geq0, v\geq0,
$$
and the equation $uv=\varphi(u)+\psi(v)$ holds whenever $u\in [q_{-}(v)sign\:v,\;q(v)sign\:v]$ or $v\in [p_{-}(u)sign\: u,\; p(u)sign\: u]$, where $p_{-}(t)=sup\{p(s): 0\leq s <t\}$, $q_{-}(t)=sup\{q(s): 0\leq s <t\}$, and $p_{-}(0)=q_{-}(0)=0$. 

We recall that an Orlicz function $\varphi$ satisfies the $\Delta_{2}$-condition($\varphi\in\Delta_{2}$) if there exist $C>0$ such that $\varphi(2t)\leq C\varphi(t)$ for all $t>0$. We further say $\varphi$ is an $N$-function whenever:
$$
\lim_{t\rightarrow 0+}\frac{\varphi(t)}{t}=0  \quad and \quad \lim_{t\rightarrow \infty}\frac{\varphi(t)}{t}=\infty.
$$

If $\omega(t)$ is a constant in an interval $A$, and for any interval $B$ such that $ A\subsetneqq B$, $\omega(t)$ is not a constant when $t\in B$. Then $A$ is called a maximal constant interval of $\omega$. $L(\omega)$ is the set composed of the maximal constant intervals.

For decreasing weight function $\omega$,
let $W(a,b)=\int_{a}^{b}\omega(t)dt$.
For arbitrary $f\in  L_{0}$, let $R(a_{1}, b_{1})=\frac{F(a_{1}, b_{1})}{W(a_{1}, b_{1})}=\frac{\int_{a_{1}}^{b_{1}}f(t)dt}{\int_{a_{1}}^{b_{1}}\omega(t)dt}$.

\begin{definition}[\cite{Halperin195305}]
	If for every $s\in (a_{1}, b_{1})$, $W(a_{1}, s)=\int_{a_{1}}^{s}\omega(t)dt>0$ and $R(a_{1},s)\leq R(a_{1}, b_{1})$, then $(a_{1}, b_{1})$ is called a level interval of $f$ with respect to $\omega$. If the level interval is not contained in a larger one, it is called a maximal level interval.
\end{definition}
The maximal level intervals are non-overlapping and denumerable.
We will introduce the definition of level function and inverse level function.
\begin{definition}\cite[Definition 3.2]{Halperin195305}, \cite[Definition 4.2]{Kaminska2014229}
	The level function of $f$ with respect to $\omega$ is denoted by $f^{0}$, and $f^{0}$ is defined by
	\begin{equation*}
		f^{0}(t)=\left\{
		\begin{aligned}
			&R(a_{n},b_{n})\omega(t) \;\;for\; t\in (a_{n}, b_{n}),\\
			&\;\;\;f(t) \;\;\;\;\; ~~~~~~otherwise,
		\end{aligned}
		\right.
	\end{equation*}
	The inverse level function of $\omega$ with respect to $f$ is defined by
	\begin{equation*}
		\omega^{f}(t) =
		\left\{
		\begin{aligned}
			&\frac{f(t)}{R(a_{n}, b_{n})}\;\;\; for \;t\in (a_{n}, b_{n}),\\
			&\;\;\omega(t) \;\;\;\; ~~~~otherwise,
		\end{aligned}
		\right.
	\end{equation*}
	where $(a_{n}, b_{n})$ is an enumeration of all maximal intervals of $f$ with respect to $\omega$, and $R(a_{n}, b_{n})=\frac{F(a_{n}, b_{n})}{W(a_{n}, b_{n})}$.
\end{definition}
We note that, for arbitrary $f\in\Lambda_{\varphi, \omega}$, we have $\omega^{f^{*}}=\omega$ since $f^{*}$ is decreasing.

For $f\in\Lambda_{\varphi, \omega}$, its modular $\rho_{\varphi, \omega}$ defined by
$$
\rho_{\varphi, \omega}(f)=\int_{0}^{\infty}\varphi(f^*(t))\omega(t)dt.
$$
The Orlicz-Lorentz function space is defined by
$$
\Lambda_{\varphi, \omega}=\{f\in L_{0}: \rho_{\varphi, \omega}(\lambda f)=\int_{0}^{\infty}\varphi(\lambda f^*(t))\omega(t)dt < \infty ~for~some~ \lambda >0\}.
$$
$\Lambda_{\varphi, \omega}$ is a Banach space under the Luxemburg norm or the Orlicz norm(\cite{Kaminska199029,Foralewski2023}). The Luxemburg norm is defined by
$$
\|f\|=\|f\|_{\varphi, \omega}=\inf\{\lambda>0: \rho_{\varphi, \omega}(\frac{f}{\lambda})\leq 1\},
$$
and the Orlicz norm is defined by
$$
\|f\|^{o}=\|f\|_{\varphi, \omega}^{o}=\sup_{\rho_{\psi, \omega}(g)\leq 1}\int_{0}^{\infty}f^*(t)g^*(t)\omega(t)dt.
$$
Moreover, the Orlicz norm admits the following Amemiya-type representation(\cite[Theorem 3.1]{Foralewski2023}), that is, 
\begin{equation}
	\|f\|_{\varphi, \omega}^{o}=\inf_{k>0}{\frac{1}{k}\left(1+\rho_{\varphi, \omega}(kf)\right)}
\end{equation}
holds for every $f\in \Lambda_{\varphi, \omega}^{o}$.
$E_{\varphi, \omega}$ denotes the subspace of $\Lambda_{\varphi, \omega}$ consisting of all $f\in L_{0}$ such that
$$
E_{\varphi, \omega}=\{f\in L_{0}: \rho_{\varphi, \omega}(\lambda f)=\int_{0}^{\infty}\varphi(\lambda f^*(t))\omega(t)dt < \infty ~for~all~ \lambda >0\}.
$$
For arbitrary $f\in \Lambda_{\varphi, \omega}$,
$\theta(f)$ is defined as follows.
\begin{equation}
	\theta(f)=\inf\left\{\lambda>0: \rho_{\varphi, \omega}(\frac{f}{\lambda})<\infty\right\}.
\end{equation}
Now we will introduce the space $\mathcal{M}_{\varphi, \omega}$. For arbitrary Orlicz function $\varphi$ and decreasing weight $\omega$ the modular $P_{\varphi, \omega}$ is given by
\begin{equation*}
	P_{\varphi, \omega}(f)=\inf\left\{\int_{R^{+}}\varphi(\frac{f^{*}}{|g|});g\prec\omega	
	\right\},
\end{equation*}
and the space $\mathcal{M}_{\varphi, \omega}$ is defined by
\begin{equation*}
	\mathcal{M}_{\varphi,\omega}=\{
	f\in L_{0}: P_{\varphi, \omega}(\lambda f)<\infty\:for ~some~\lambda >0
	\}
\end{equation*}
Let $\mathcal{M}_{\varphi, \omega}=(M_{\varphi, \omega}, \|\cdot\|_{\varphi, \omega})$ and $\mathcal{M}_{\varphi, \omega}^{o}=(\mathcal{M}_{\varphi, \omega}^{o})$, where $\|\cdot\|_{\varphi, \omega}$ and $\|\cdot\|_{\mathcal{M}_{\varphi, \omega}^{o}}$ are defined by
\begin{align}
	\|f\|_{\mathcal{M}_{\psi, \omega}^{o}}
	&=\inf_{k>0}\left\{ \frac{1}{k}\left(P_{\psi, \omega}(kf)+1\right)\right\},\\
	\|f\|_{\mathcal{M}_{\psi, \omega}}
	&=\inf\left\{\lambda>0: \: P_{\psi, \omega}\left(\frac{f}{\lambda}\right)\leq 1\right\}.
\end{align}
\section{Auxiliary results}\label{Auxiliary}

Let $\Lambda_{\varphi, \omega}=(\Lambda_{\varphi, \omega}^{o}, \|\cdot\|_{\varphi, \omega})$, 
$\Lambda_{\varphi, \omega}^{o}=(\Lambda_{\varphi, \omega}, \|\cdot\|_{\varphi, \omega})$.

\begin{lemma}\cite[Theorem 8.10]{2019abstractlorentz},\cite[Theorem 2.2]{Kaminska2014229}
	For arbitrary Orlicz function $\varphi$ and decreasing weight $\omega$, \\
	(1) the k\"{o}the dual of Orlicz-Lorentz spaces $\Lambda_{\varphi, \omega}$ and $\Lambda_{\varphi, \omega}^{o}$ are expressed as
	\begin{align*}
		(\Lambda_{\varphi, \omega}, \|\cdot\|_{\varphi, \omega})^{\prime}
		&=(\mathcal{M}_{\psi, \omega}, \|\cdot\|_{\mathcal{M}_{\psi, \omega}^{o}})\\
		(\Lambda_{\varphi, \omega}, \|\cdot\|_{\varphi, \omega}^{o})^{\prime}
		&=(\mathcal{M}_{\psi, \omega}, \|\cdot\|_{\mathcal{M}_{\varphi, \omega}})
	\end{align*}
	with equality of corresponding norms.\\
	(2) If $\varphi\in\Delta_{2}$ and $\int_{0}^{\infty}\omega(t)dt = W(\infty)=\infty$. Then the dual spaces $(\Lambda_{\varphi, \omega})^{*}$ and $(\Lambda_{\varphi, \omega}^{o})^{*}$ are isometrically isomorphic to their corresponding K\"{o}the dual spaces. For functional $L_{\phi}\in (\Lambda_{\varphi, \omega})^{*}(resp., \Phi\in (\Lambda_{\varphi, \omega}^{o})^{*})$, there exists $\phi\in \mathcal{M}_{\psi, \omega}^{o}(resp., \phi\in \mathcal{M}_{\psi, \omega}^{o}$ such that
	$$
	\Phi(f)=L_{\phi}(f)=\int_{0}^{\infty}f(t)\phi(t)dt,\;\; f\in \Lambda_{\varphi, \omega}.
	$$
	and $\|\Phi\|_{(\Lambda_{\varphi, \omega})^{*}}=\|\phi\|_{\mathcal{M}_{\psi, \omega}}^{o}\;(resp., \|\Phi\|_{(\Lambda_{\varphi, \omega}^{o})^{*}}=\|\phi\|_{\mathcal{M}_{\psi, \omega}})$.
\end{lemma}
\begin{lemma}\label{0818}\cite[Lemma 6.15, Theorem 8.9]{2019abstractlorentz}
 Assume $\varphi$ is an arbitrary Orlicz function, $\omega$ is a decreasing positive weight on $R^{+}$ and $\int_{0}^{t}\omega(t)dt<\infty$ for every $t<\infty$. Then 
 $(\Lambda_{\varphi, \omega})^{\prime \prime}=(Q_{L_{\psi}, \omega}^{o})^{\prime}
 =(\mathcal{M}_{\psi, \omega}^{o})^{\prime}=\Lambda_{\varphi, \omega}$, where 
 $Q_{L_{\psi}, \omega}^{o}=(Q_{L_{\psi}, \omega}, \|\cdot\|_{Q_{L_{\psi}, \omega}}^{o})$ is defined in \cite[Definition 6.1]{2019abstractlorentz}.
\end{lemma}

\begin{lemma} \cite[Theorem 2.5]{Kaminska199029},
	For  $f\in\Lambda_{\varphi, \omega}$, the following conditions are equivalent:\\
	(1)$\varphi \in \Delta_{2}$,\\
	(2)For arbitrary $\varepsilon>0$, there exists $\delta>0$, such that $\rho_{\varphi, \omega}(f)\geq 1-\varepsilon$ whenever $\|f\|_{\varphi, \omega}\geq 1-\delta$,\\
	(3)$\rho_{\varphi, \omega}(f)=1$ if and only if $\|f\|_{\varphi, \omega}=1$,\\
	(4)$\rho_{\varphi, \omega}(f_{n})\rightarrow 0$ if and only if $\|f_{n}\|_{\varphi, \omega}\rightarrow 0\:(n\rightarrow\infty)$.
\end{lemma}
For $x\in\Lambda_{\varphi, \omega}^{o}$, define
\begin{align*}
	&k^{* }=k^{* }(x)=\inf\{k>0: \rho_{\psi, \omega}(p(kx))\geq 1\}.\\
	&k^{**}=k^{**}(x)=\sup\{k>0: \rho_{\psi, \omega}(p(kx))\leq 1\}.\\
	&K(x)=[K^{*}(x), k^{**}(x)]
\end{align*}
\begin{lemma}\cite[Theorem 3.1, Theorem 3.2]{Foralewski2023}
	Let $\varphi$ be an Orlicz function and $\omega$ be a decreasing weight. $x\in\Lambda_{\varphi, \omega}^{o}$ and $K(x)\neq \emptyset$.
	Then $\|x\|^{o}_{\varphi, \omega}=\frac{1}{k}\left( 1+\rho_{\varphi, \omega}(kx) \right)$ if and only if $k\in K(x)$.
\end{lemma}
\begin{lemma}\cite[Theorem 2.8]{Wang2024}
	For arbitrary $f\in \mathcal{M}_{\varphi, \omega}^{o}$,
	if there exists $k>0$ such that
	$$
	\int_{0}^{\infty}\psi(p(\frac{kf^{*}(t)}{\omega^{f^{*}}}))\omega^{f^{*}}(t)=1.
	$$
	Then
	$$
	\|f\|_{\mathcal{M}_{\varphi, \omega}}^{o}=\frac{1}{k}\left(1+P_{\varphi, \omega}(kf)\right).
	$$
\end{lemma}

For arbitrary $f\in \mathcal{M}_{\varphi, \omega}$, define
\begin{align*}
	k^{*}_{\mathcal{M}}=k^{*}_{\mathcal{M}}(v) &:=\inf\{k>0: \int_{0}^{\infty}\psi(p(\frac{kf^{*}(t)}{\omega^{f^{*}}(t)}))\omega^{f^{*}}(t)dt\geq 1\},\\
	k^{**}_{\mathcal{M}}=k^{**}_{\mathcal{M}}(v)  &:=\sup\{k>0: \int_{0}^{\infty}\psi(p(\frac{kf^{*}(t)}{\omega^{f^{*}}(t)}))\omega^{f^{*}}(t)dt\leq 1\}.
\end{align*}
Obviously, $k^{*}_{\mathcal{M}}\leq k^{**}_{\mathcal{M}}$. Let $K_{\mathcal{M}}=[k_{\mathcal{M}}^{*}, k^{**}_{\mathcal{M}}]$. If $k^{**}_{\mathcal{M}}<\infty$,
then $K_{\mathcal{M}}\neq \emptyset$.

\begin{lemma}\cite[Theorem 9, p. 36]{kan1982},
	For any $f\in (\Lambda_{\varphi, \omega}^{o})^{*}$, f has a unique decomposition
	$$
	f=L_{g}+s,\;\;g\in \mathcal{M}_{\psi, \omega},\: s\in F
	$$
	where s is singular and $F$ denotes the set of singular functionals of $x$.
\end{lemma}

\begin{lemma}\cite[Theorem 2.48]{Chen1996}
	For $x\in \Lambda_{\varphi, \omega}$ satisfying $\theta(x)\neq 0$,  there exist two singular functionals $s_{1}$ and $s_{2}$ such that  $s_{1}\neq s_{2}$ and $s_{1}(x)=s_{2}(x)=\theta(x)$.
\end{lemma}

For $f\in(\Lambda_{\varphi, \omega})^{*}$, let  $\|\cdot\|=\|\cdot\|_{(\Lambda_{\varphi, \omega}^{o})^{*}}$ and $\|\cdot\|^{o}=\|\cdot\|_{(\Lambda_{\varphi, \omega})^{*}}$.
\begin{lemma}\cite[Corollary 1.49]{Chen1996}
	$\|f\|=\|f\|^{o}$ if and only if $f\in F$.
\end{lemma}

\begin{lemma}\cite[Lemma 2.22]{Wang2024}
	For any $f\in (\Lambda_{\varphi, \omega})^{*}$, $f=L_{v}+s$, we have $\|f\|=\|v\|_{\mathcal{M}_{\psi, \omega}}+\|s\|$, $\|f\|^{o}= \|v\|_{\mathcal{M}_{\psi, \omega}^{o}}+\|s\|^{o}$\\
\end{lemma}

\begin{lemma}\cite[Theorem 1.43, Theorem 1.44]{Chen1996}, \cite[Theorem 2.21]{Wang2024}
	For any $f\in \Lambda_{\varphi, \omega}$, we have $d(f)=d^{o}(f)=\theta(f)$, where
	$$
	d(f)=\inf\{\|f-f_{e}\|: \: f_{e}\in E_{\varphi,\omega}\};\;d^{o}(f)=\inf\{\|f-f_{e}\|_{\varphi, \omega}^{o}:\: f_{e}\in E_{\varphi, \omega}\}.
	$$
\end{lemma}

\begin{lemma}\cite[Chapter 2, Proposition 1.7]{Bennett1953}
	\label{converge}
	For $f\in L_{0}$ and $f_{n}\in L_{0}$,
	we have
	\begin{equation*}
		|f|\leq \liminf_{n\rightarrow\infty} |f_{n}|\;\mu-a.e. \Rightarrow f^{*}\leq \liminf_{n\rightarrow} f_{n}^{*};
	\end{equation*}
	In particular,
	$$
	|f_{n}|\uparrow |f|\: \mu-a.e. \Rightarrow f^{*}\leq \liminf_{n\rightarrow\infty}f_{n}^{*}
	$$
\end{lemma}

\begin{lemma}\cite[Remark 1, Case 4]{WangSEP2023}\label{norm123}
	For Orlicz-Lorentz space $\Lambda_{\varphi, \omega}^{o}$. 
	Assume $\lim_{u\rightarrow\infty}\frac{\varphi(u)}{u}=B<\infty$ and let $x\in \Lambda_{\varphi, \omega}\backslash \{0\}$. 
	If $\psi(B)\int_{0}^{m(supp~x)}\omega(t)dt\leq 1$, then $k^{**}(x)=\infty$ and 
	\begin{equation*}
		\|x\|_{\varphi, \omega}^{o}=B\int_{0}^{\infty}x^{*}(t)\omega(t)dt=B\|x\|_{\Lambda_{\omega}}
	\end{equation*} 
\end{lemma}

\section{Main results}

\subsection{
	Norm expression of $\|\cdot\|_{\mathcal{M}_{\varphi, \omega}}$
}
\label{Normexpression}
\begin{lemma}\cite{2019abstractlorentz}
	Let $\varphi$ be an Orlicz function and $f\in L_{0}$ such that $P_{\varphi, \omega}(f)<\infty$.
	Then the inverse function $\omega^{f^{*}}$ satisfies $\omega^{f^{*}}\prec \omega$ and
	\begin{align*}
		&P_{\varphi, \omega}(f)=\int_{0}^{\infty}\varphi(\frac{(f^{*})^{0}}{\omega})\omega(t)dt
		=\int_{0}^{\infty}\varphi(\frac{f^{*}}{\omega^{f^{*}}})\omega^{f^{*}}(t)dt,\\
		&\rho_{\psi, \omega}\left(p\left(\frac{(f^{*})^{0}}{\omega}\right)\right)
		=\int_{0}^{\infty}\psi\left(p\left(\frac{f^{*}}{\omega^{f^{*}}}\right)\right)
		\omega^{f^{*}}(t)dt.
	\end{align*}
\end{lemma}
We now investigate the Orlicz norm in $\mathcal{M}_{\varphi, \omega}^{o}$ generated by arbitrary Orlicz function.
\begin{theorem}\label{Norminkothedual}
	Assume $\varphi$ is an Orlicz function and $\omega$ is a decreasing weight. Let $v\in\mathcal{M}_{\varphi, \omega}^{o}$. \\
	(1)If Orlicz function $\varphi$ is an $N$-function, then $k^{**}_{\mathcal{M}}<\infty$\\
	(2)If $\varphi$ satisfies $\lim_{u\rightarrow \infty}$ $\frac{\varphi(t)}{t}=B<\infty$ and $\lim_{t\rightarrow\infty}Bt-\varphi(t)=\infty$, then $k^{**}_{\mathcal{M}}(v)<\infty$.\\
	(3)If $\varphi$ satisfies $\lim_{t\rightarrow\infty}\frac{\varphi(t)}{t}=B<\infty$ and $\lim_{t\rightarrow \infty}Bt-\varphi(t)<\infty$, $k\in\mathcal{M}_{\varphi, \omega}^{o}$ satisfies $\psi(B)\int_{0}^{\mu(supp~v)}\omega(t)dt> 1$, then $k_{\mathcal{M}}^{**}(v)<\infty$.\\
	(4)If $\varphi$ satisfies $\lim_{t\rightarrow\infty}\frac{\varphi(t)}{t}=B<\infty$ and $\lim_{t\rightarrow\infty}Bt-\varphi(t)<\infty$. $v\in\mathcal{M}_{\varphi, \omega}^{o}$ and $\psi(B)\int_{0}^{\mu(supp~v)}\omega(t)dt\leq 1$, then $k^{**}_{\mathcal{M}}(v)=\infty$.\\
	In case (1)(2)(3),
	$\|v\|_{\mathcal{M}_{\varphi, \omega}^{o}}=\frac{1}{k}(1+P_{\varphi, \omega}(kx))$ where $v\in K_{\mathcal{M}}(v)$.
	In case (4), 
	$\|v\|_{\mathcal{M}_{\varphi, \omega}^{o}}=B\int_{0}^{\mu(supp~v)}v^{*}(t)dt
	=B\int_{0}^{\infty}v^{*}(t)dt$.
\end{theorem}
\begin{proof}
	(1)The proof is in \cite[Theorem 2.11]{Wang2024}.\\
		(2)
		The condition $\lim_{t\rightarrow\infty}Bt-\varphi(t)=\infty$ is equivalent to $\psi(B)=\infty$. For all $v\in\mathcal{M}_{\varphi, \omega}^{o}$ and $v\neq 0$, from \cite[Theorem 3.6]{Halperin195305} there exists $k_{v}$ and $t_{v}$ such that
		\begin{equation*}
			\int_{0}^{t_{v}}\psi\left(p\left(\frac{k_{v}(v^{*})^{0}(t_{v})}{\omega(t_{v})}\right)\right)\omega(t_{v})dt\geq1
		\end{equation*}
		thus $k^{*}_{\mathcal{M}}(v)\leq k^{**}_{\mathcal{M}}(v)\leq k_{v}$.\\
		(3)
		Since $\psi(B)\int_{0}^{\mu (supp~v)}\omega(t)dt>1$, 
		there exists $a\in R^{+}$, $0<a<\mu(supp~v)$, such that
		\begin{equation*}
			\psi(B)\int_{0}^{a}\omega(t)dt\geq 1
		\end{equation*}
		Since $\lim_{u\rightarrow\infty}p(u)=B$, there exists $u_{v}>0$ for which $\psi(p(u_{v}))\int_{0}^{a}\omega(t)dt>1$. Define $k^{\prime}_{v}=\frac{u_{v}\omega(a)}{(v^{*})^{o}(a)}$ then
		\begin{equation*}
			\int_{0}^{\infty}\psi(p(\frac{u_{v}\omega(a)}{(v^{*})^{0}(a)}
			\frac{(v^{*})^{0}(t)}{\omega(t)}))\omega(t)dt\geq \int_{0}^{\infty}\psi(p(u_{v}))\omega(t)dt> 1
		\end{equation*}
		Consequently, $k^{*}_{\mathcal{M}}(v)\leq k^{**}_{\mathcal{M}}(v)\leq k_{v}$.\\
		(4)
		For arbitrary $k>0$, Since $p(\frac{kv^{*}(t)}{\omega_{v}^{*}}(t))\leq B$. Then
		\begin{equation*}
			\int_{0}^{\infty}\psi(p(\frac{kv^{*}(t)}{\omega_{v}^{*}(t)}))\omega_{v}^{*}(t)dt\leq \psi(B)\int_{0}^{\mu (supp~v)}\omega(t)dt\leq 1. 
		\end{equation*}
		Thus $K_{\mathcal{M}}^{**}(v)=\infty$. 

		Since $\psi(B)<\infty$, $\psi(u)=\infty$ for every $u>B$.
		Then for any $x\in \mathcal{M}_{\varphi, \omega}$ such that $\rho_{\psi, \omega}(x)\leq 1$, we get that $x(t)\leq B$ a.e. in $R^{+}$ and
		\begin{equation*}
			\int_{0}^{\infty}x(t)v(t)dt\leq \int_{0}^{\infty}x^{*}(t)v^{*}(t)dt\leq B\int_{0}^{\infty}v^{*}(t)dt
		\end{equation*}
		For $x_{0}=B\chi_{[0, \mu(supp~v)]}$ we have $\rho_{\psi, \omega}(x_{0})=\psi(B)\int_{0}^{\mu(supp~v)}\omega(t)dt\leq 1$ and
		\begin{equation*}
			\int_{0}^{\infty}x(t)v(t)dt=\int_{0}^{\infty}x^{*}(t)v^{*}(t)dt=B\int_{0}^{\mu(supp~v)}v^{*}(t)dt
		\end{equation*}
		Thus $\|v\|_{\mathcal{M}_{\varphi, \omega}^{o}}=
		B\int_{0}^{\mu(supp~v)}v^{*}(t)dt$.
	\end{proof}

\subsection{Supporting functionals}\label{Supportingfunctional}
\begin{theorem}
	Let $\psi$ be an Orlicz function satisfying
	$\lim_{u\rightarrow\infty}\frac{\psi(u)}{u}=B<\infty$ and $\lim_{u\rightarrow\infty}Bu-\psi(u)<\infty$.
	If $v\in\mathcal{M}_{\psi, \omega}^{o}$ attain its norm at $x\in \Lambda_{\varphi,\omega}$ and $\psi(B)\int_{0}^{\mu(supp~v)}\omega(t)dt\leq 1$, then $\|x\|_{\varphi, \omega}=\frac{1}{B}x^{*}(0)$.
\end{theorem}
\begin{proof}
	If $\psi$ satisfies the hypothesis, then $(\Lambda_{\varphi, \omega})^{*}$ 
is isometrically isomorphic to the K\"{o}the dual. 
	Assume $v\in \mathcal{M}_{\psi, \omega}^{o}$ and $\|v\|_{\mathcal{M}_{\psi, \omega}^{o}}=1$, then $B\int_{0}^{\infty}v^{*}(t)dt=1$. 
	From Lemma \ref{0818} and part (4) of Theorem \ref{Norminkothedual}  we have
	\begin{align*}
		\|x\|_{\Lambda_{\varphi, \omega}}
		&=\|x\|_{(Q^{o}_{L_{\varphi}, \omega})^{\prime}}\\
		&=\sup\left\{\int_{0}^{\infty} v    (t)x    (t)dt:~\|v\|_{\mathcal{M}_{\psi, \omega}^{o}}=1\right\}\\
		&=\sup\left\{\int_{0}^{\infty} v^{*}(t)x^{*}(t)dt:~\|v\|_{\mathcal{M}_{\psi, \omega}^{o}}=1\right\}\\
		&=\sup\left\{\int_{0}^{\infty} v^{*}(t)x^{*}(t)dt: ~B\int_{R^{+}}v^{*}(t)dt=1 \right\}\\
		&=B\|\frac{v^{*}}{\omega^{v^{*}}}\|_{\Lambda_{\omega^{v^{*}}}}
		\cdot\frac{1}{B}\|x^{*}(t)\omega^{v^{*}}(t)\|_{M_{W^{v^{*}}}}\\
		&=B\int_{R^{+}}v^{*}(t)dt\cdot  \frac{1}{B}\|x^{*}(t)\omega^{v^{*}}(t)\|_{M_{W^{v^{*}}}}\\
		&=\frac{1}{B}\|x^{*}(t)\omega^{v^{*}}(t)\|_{M_{W^{v^{*}}}}\\
		&=\frac{1}{B}\sup_{\alpha\in R^{+}}\frac{\int_{0}^{\alpha}x^{*}(t)\omega^{v^{*}}(t)dt}{\int_{0}^{\alpha}\omega^{v^{*}}(t)dt}\\
		&=\frac{1}{B}x^{*}(0).
	\end{align*}
\end{proof}
\begin{theorem}
	Let $\omega$ be an arbitrary decreasing weight. Assume $\psi$ satisfies   $\lim_{t\rightarrow\infty}\frac{\psi(t)}{t}=B<\infty$ and $\lim_{t\rightarrow\infty}Bt-\psi(t)<\infty$. 
	For $v\in\mathcal{M}_{\psi, \omega}^{o}$ and $\psi(B)\int_{0}^{\mu(supp~v)}\omega(t)dt\leq 1$, then $v$ attains its norm at $x\in S(\Lambda_{\varphi, \omega})$ if and only if \\
	(1)$\rho_{\varphi, \omega}(x)=1$ and\\
	(2)$x(t)=B\chi_{supp~v}(t)+x_{0}(t)$ where $(supp~x_{0})\cap (supp~v)=\emptyset$ and $|x_{0}(t)|\leq B$.
\end{theorem}
\begin{proof}
	\textsl{Necessity:}
	Without loss of generality, we can assume $v\geq 0$.
	Since $\psi$ satisfies the hypothesis, then $\psi\in\Delta_{2}$, and it is equivalent to $\varphi\in\Delta_{2}$. Therefore $x\in S(\Lambda_{\varphi, \omega})$ is equivalent to $\rho_{\varphi, \omega}(x)=1$.
	Let $v\in \mathcal{M}_{\psi, \omega}$ and $v$ attains its norm at $x\in S(\Lambda_{\varphi, \omega})$, then $x^{*}(0)=B$ and $x(t)\leq B$ when $t\neq 0$. Thus
	\begin{align*}
		\|v\|_{\mathcal{M}_{\psi,\omega}^{o}}
		&=\int_{R^{+}}x(t)v(t)dt\\
		&\leq B\int_{supp~v}v(t)dt\\
		&=B\int_{0}^{\mu(supp~v)}v^{*}(t)dt\\
		&=\|v\|_{\mathcal{M}_{\psi, \omega}^{o}}
	\end{align*}
	Therefore $x\chi_{supp~v}(t)=B\chi_{supp~ v}$.\\
	\textsl{Sufficiency}:
	Assume $x(t)=B\chi_{supp~v}(t)+ x_{0}(t)$, therefore
	\begin{equation*}
		\int_{R^{+}}x(t)v(t)dt=B\int_{supp~v}v(t)dt=B\int_{0}^{\mu(supp~v)}v^{*}(t)dt=\|v\|_{\mathcal{M}_{\psi, \omega}^{o}}.
	\end{equation*}
\end{proof}
\begin{theorem}\label{normattainluxemburg}For arbitrary Orlicz function $\varphi$ and decreasing weight $\omega$, $f\in \Lambda_{\varphi, \omega}^{*}$. Assume 
	$f=L_{v}+s$, where $0\neq v\in \mathcal{M}_{\psi, \omega}^{o}$,  $K_{\mathcal{M}}(v)\neq \emptyset$, $s\in F$. Then $f$ attains its norm at $x\in S(\Lambda_{\varphi, \omega})$ if and only if\\
	(1)$\int_{0}^{\infty} x(t)v(t)dt=\int_{0}^{\infty}x^{*}(t)v^{*}(t)dt$ \\
	(2)$\rho_{\varphi, \omega}(x)=1$.\\
	(3)$s(x)=\|s\|$.\\
	(4)$\int_{0}^{\infty}kv^{*}(t)x^{*}(t)dt=\rho_{\varphi, \omega}(x)+P_{\psi,\omega}(kv)$, $k\in K_{\mathcal{M}}(v)$.
\end{theorem}
\begin{proof}
	The proof is almost the same as that of Theorem 3.1 in \cite{Wang2024}, since 
	$x$ has the same norm expression.
	\end{proof}
\begin{lemma}\cite[Theorem 2.48]{Chen1996}
	Let $x\in \Lambda_{\varphi, \omega}$, and $\theta(x)\neq 0$. Then there exist two singular functionals $s_{1}$ and $s_{2}$ such that $s_{1}\neq s_{2}$ and $s_{1}(x)=s_{2}(x)=\theta(x)$.
\end{lemma}

\begin{theorem}\cite[Theorem 3.3]{Wang2024}
	Let $x\in \Lambda_{\varphi, \omega}$, then $Grad(x) \subset\mathcal{M}_{\psi, \omega}^{o}$ if and only if $\theta(x)<1$.
\end{theorem}

\begin{theorem}\cite[Theorem 3.4]{Wang2024}
	Let
	$\rho_{\varphi, \omega}(\frac{|x(t)|}{\|x\|})=1$. Then $v\in S(\mathcal{M}_{\varphi, \omega}^{o})$ is a supporting functional of $x$ if and only if the following conditions are satisfied\\
	(1) $v=\frac{\varpi}{\|\varpi\|_{\mathcal{M}_{\psi, \omega}}^{o}}$ for some $\varpi$ satisfying
	$$
	p_{-}(\frac{x^{*}(t)}{\|x\|_{\varphi, \omega}})\omega(t)\leq \varpi^{*}(t)  \leq p(\frac{x^{*}(t)}{\|x\|_{\varphi, \omega}})\omega(t),\; a.e. \:on\: R_{+},
	$$
	(2) $\int_{0}^{\infty}x^{*}(t)v^{*}(t)dt=\int_{0}^{\infty} x(t)v(t)dt$
\end{theorem}
Then we come to $\Lambda_{\varphi, \omega}^{o}$. 

\begin{theorem}
	For Orlicz function $\varphi$ such that $\lim_{t\rightarrow\infty}\frac{\varphi(t)}{t}=B<\infty$ and $\lim_{t\rightarrow\infty}Bt-\varphi(t)<\infty$. Let $x\in S(\Lambda_{\varphi, \omega}^{o})$ satisfy $\psi(B)\int_{0}^{\mu(supp\: x)}\omega(t)dt\leq 1$.
	If $v\in \mathcal{M}_{\psi, \omega}$ attains its norm at $x$, then\\
	\begin{equation*}
		\|v\|_{\mathcal{M}_{\psi, \omega}}=\frac{1}{B}\|v\|_{M_{W}}=\frac{1}{B}\sup_{\alpha\in R^{+}}\frac{1}{W(\alpha)}\int_{0}^{\alpha}v^{*}(t)dt.
	\end{equation*}
\end{theorem}
\begin{proof}
	From lemma \ref{norm123}
	\begin{align*}
		\|v\|_{\mathcal{M}_{\psi, \omega}}
		&=\sup\left\{\int_{0}^{\infty} x(t)v(t)dt: \|x\|_{\Lambda_{\varphi, \omega}^{o}}=1\right\}\\
		&=\frac{1}{B}\sup\left\{\int_{0}^{\infty} Bx(t)v(t)dt: \|Bx\|_{\Lambda_{\omega}}=1\right\}\\
		&=\frac{1}{B}\|v\|_{M_{W}}\\
		&=\frac{1}{B}\sup\frac{1}{W(\alpha)}\int_{0}^{\alpha}v^{*}(t)dt.
	\end{align*}
\end{proof}

\begin{theorem}\label{12}
	For Orlicz function $\varphi$ such that $\lim_{t\rightarrow\infty}\frac{\varphi(t)}{t}=B<\infty$ and $\lim_{t\rightarrow\infty}Bt-\varphi(t)<\infty$. Let $x\in S(\Lambda_{\varphi, \omega}^{o})$ satisfying $\psi(B)\int_{0}^{\mu(supp\: x)}\omega(t)dt\leq 1$.
	$f=L_{v}+s$ attains its norm at $x$ if and only if \\
	(1) $s=0$.\\
	(2) $\int_{0}^{\infty} x(t)v(t)dt=\int_{0}^{\infty} x^{*}(t)v^{*}(t)dt$.\\
	(3) $v^{*}(t)= B\|f\|\omega(t)$ on $(0, supp\: x).$ \\
\end{theorem}
\begin{proof}
	\textsl{Necessity. }
	Assume $f=L_{v}+s$ is norm-attainable at $x$.
	Since $x\in E_{\varphi, \omega}^{o}$, then $s=0$.
	If (2) is not satisfied, then $\int_{R^{+}}x(t)v(t)dt<\int_{R^{+}}x^{*}(t)v^{*}(t)dt$ and
	\begin{align*}
		\|f\|=f(x)
		&=\int_{0}^{\infty}x(t)v(t)dt + s(x)\\
		&=\int_{R^{+}}x(t)v(t)dt\\
		&<    \int_{0}^{\infty}Bx^{*}(t)\frac{v^{*}(t)}{B}dt \\
		&\leq \|Bx(t)\|_{\Lambda_{\omega}}\|\frac{v}{B}\|_{M_{W}}\\
		&\leq \|v(t)\|_{\mathcal{M}_{\psi, \omega}}\\
		&\leq \|f\|,
	\end{align*}
	a contradiction.
	Since $\|v\|_{\mathcal{M}_{\psi, \omega}}=\|f\|-\|s\|=\|f\|$, then it is easy to verify $v\prec B\|f\| \omega$.
	Therefore, for arbitrary $c>0$, $\int_{0}^{c} B\|f\|\omega(s)-v^{*}(s)ds \geq 0$.
	Since $f$ attain its norm at $x$, we have
	\begin{equation*}
		\begin{aligned}
			\|f\|
			&=f(x)=\int_{0}^{\infty}x(t)v(t)dt\\
			&=\int_{0}^{\infty} x^{*}(t)v^{*}(t)dt\\
			&\leq B\|f\|\int_{0}^{\infty} x^{*}(t)\omega(t)dt\\
			&=\|f\| 
		\end{aligned}
	\end{equation*}
	Then 
	\begin{equation*}
		\int_{0}^{\infty} x^{*}(t)(B\|f\|\omega(t)-v^{*}(t))dt=0
	\end{equation*}
	which implies 
	\begin{equation*}
		v^{*}(t)=B\|f\|\omega(t), a.e. on ~supp\: x
	\end{equation*}
	\textsl{Sufficiency. }
	Assume $f=L_{v}+s$ satisfying condition(1)(2)(3). Then
	\begin{align*}
		f(x)
		&=\int_{0}^{\infty}v(t)x(t)dt \\
		&=\int_{0}^{\infty}v^{*}(t)x^{*}(t)dt \\
		&=B\|f\|\int_{0}^{\infty}x^{*}(t)\omega(t)dt \\
		&=\|f\|.
	\end{align*}
	where the last equality follows from the fact that $\|x\|_{\Lambda^{o}_{\varphi, \omega}}=1$.
	Finally, $f$ is norm-attainable at $x$.
\end{proof}

\begin{theorem}
	Assume $\varphi$ satisfies the assumption in Theorem \ref{12}. Then 
	$v\in Grad(x)$ if and only if $v^{*}=B\omega$ a.e., and $\int_{0}^{\infty} x(t)v(t)dt=\int_{0}^{\infty}x^{*}(t)v^{*}(t)dt$
\end{theorem}

\begin{proof}
	Since $\int_{R^{+}}\omega(t)dt=\infty$, therefore $x\in E_{\varphi, \omega}^{o}$. For $f\in S(\Lambda_{\varphi, \omega}^{*})$, $f=L_{v}+s$. Then $s=0$. From theorem \ref{12} we  have the proof.
\end{proof}
\begin{theorem}\label{NormattainOrlicznorm}\cite[Theorem 3.1]{Wang2024}
	Let $\varphi$ be an Orlicz function and $\omega$ be a decreasing weight.
	$f=L_{v}+s(0\neq v\in \mathcal{M}_{\psi, \omega}, s\in F)$ attains its norm at $x\in S(\Lambda_{\varphi, \omega}^{o})$ ($K(x)\neq \emptyset$) if and only if\\
	(1)$\int_{0}^{\infty}x(t)v(t)=\int_{0}^{\infty}x^{*}(t)v^{*}(t)$ \\
	(2)$s(kx)=\|s\|$.\\
	(3)$P_{\varphi, \omega}(\frac{v}{\|f\|})+\frac{\|s\|}{\|f\|}=1$.\\
	(4)$\int_{0}^{\infty}\frac{kv(t)x(t)}{\|f\|}dt=\rho_{\varphi, \omega}(kx)+P_{\psi,\omega}(\frac{v}{\|f\|})$ where
	$k\in K(x)$.\\
\end{theorem}

\begin{theorem}\label{th14}\cite[Theorem 3.7]{Wang2024}
	For arbitrary Orlicz function $\varphi$ and decreasing weight $\omega$. Let $x\in S(\Lambda_{\varphi, \omega}^{o})$ and $K(x)\neq \emptyset$. Then $Grad(x)\subset \mathcal{M}_{\psi, \omega}$ 
	if and only if one of the following conditions is satisfied.\\
	(1) $\theta(kx)<1$, for arbitrary $k\in K(x)$.\\
	(2) $P_{\psi,\omega}(p_{-}(kx^{*})\omega)=\rho_{\psi, \omega}(p_{-}(kx))=1$.
\end{theorem}

\begin{theorem}\cite[Theorem 3.8]{Wang2024}\label{Th15}
	Assume $\varphi$ is an Orlicz function and $\omega$ is a decreasing weight. $L_{v}\in S(\mathcal{M}_{\psi, \omega})$ is a supporting functional of $x\in \Lambda_{\varphi, \omega}^{o}$($K(x)\neq \emptyset$) if and only if the following conditions are satisfied.\\
	(1) $P_{\psi, \omega}(v)=1$, $\int_{0}^{\infty} x(t)v(t)dt=\int_{0}^{\infty}x^{*}(t)v^{*}(t)dt$. \\
	(2) $p_{-}(kx^{*}(t))\omega(t)\leq v^{*}(t)\leq p(kx^{*}(t))\omega(t)$ a.e on $R^{+}$ for every $k\in K(x)$.
\end{theorem}

\nocite{Halperin195305}\nocite{Ryff1970449}
\section{Acknowledgement}
The authors gratefully acknowledge financial support from China Scholarship Council and the National Natural Science Foundation of P. R. China (Nos. 11971493).\\
\section*{Author Contribution} 
All authors have equally contributed to writing and preparing the manuscript text. All authors reviewed the results and approved the final version of the manuscript.\\
\section*{Data Availability} 
No data were used to support this study.
\section*{Conflicts of Interest}
The authors declare that there is no conflict of interest regarding the publication of this paper.

\bibliographystyle{amsplain}

\end{document}